\documentclass{article}

\usepackage{amsmath,amsthm,amssymb,bbm,leftidx,color}

\usepackage{lmodern} 
\usepackage[T1]{fontenc}

\definecolor{e-mail}{rgb}{0,.40,.80}
\definecolor{reference}{rgb}{.20,.60,.22}
\definecolor{citation}{rgb}{0,.40,.80}

\renewenvironment{abstract}{\footnotesize\par\noindent\ignorespaces{\itshape Abstract}\quad}

\newcommand{\keywords}[1]{{\footnotesize\noindent\textit{Keywords:}
  \parbox[t]{120mm}{\raggedright\footnotesize#1}}\vspace{.5pc}}

\newcommand{\classification}[2]{{\footnotesize\noindent 2010 \textit{Mathematics subject
classification:} Primary #1\\[-2pt]\phantom{\footnotesize\noindent 2010
\textit{Mathematics subject classification:}} Secondary #2\vspace{0pc}}}

\theoremstyle{plain}
\newtheorem{theorem}{Theorem}[section]
\newtheorem{lemma}[theorem]{Lemma}
\newtheorem{proposition}[theorem]{Proposition}

\theoremstyle{definition}
\newtheorem{definition}[theorem]{Definition}

\newtheorem{example}[theorem]{Example}

\theoremstyle{remark}
\newtheorem{remark}[theorem]{Remark}

\title{\large\sc UNIPOTENT DIFFERENTIAL ALGEBRAIC GROUPS AS PARAMETERIZED DIFFERENTIAL GALOIS GROUPS\footnote{A.~Ovchinnikov was supported by the NSF grant CCF-0952591; M.~F.~Singer was supported by the NSF grant CCF-1017217}}
\author{\normalsize\sc ANDREY MINCHENKO, ALEXEY OVCHINNIKOV, \\ \normalsize AND MICHAEL F. SINGER\\
\\
\it
\normalsize The Hebrew University of Jerusalem, Einstein Institute of Mathematics\\
\it
\normalsize Jerusalem, 91904, Israel {\rm ({\color{blue}an.minchenko@gmail.com})}\\ \\
\it
\normalsize CUNY Queens College, Department of Mathematics\\
\it
\normalsize65-30 Kissena Blvd, Queens, NY 11367, USA\\
\it
\normalsize CUNY Graduate Center, Department of Mathematics \\
\it
\normalsize 365 Fifth Avenue, New York, NY 10016, USA {\rm({\color{blue} aovchinnikov@qc.cuny.edu})}\\ \\
\it
\normalsize North Carolina State University, Department of Mathematics\\
\it
\normalsize Raleigh, NC 27695-8205, USA {\rm ({\color{blue}singer@ncsu.edu})}
\vspace{-0.1in}
}

\date{}

\usepackage[pdfstartview=FitH,
          colorlinks,
           linkcolor=reference,
           citecolor=citation,
            urlcolor=e-mail,
           ]{hyperref}

\def\C{{\mathbb C}}
\def\Q{{\mathbb Q}}
\def\p{{\mathfrak p}}

\DeclareMathOperator{\Const}{\mathcal C}
\DeclareMathOperator{\I}{\mathbb I}
\DeclareMathOperator{\U}{\mathcal U}
\DeclareMathOperator{\Mn}{\bf M}

\DeclareMathOperator{\SL}{SL}
\DeclareMathOperator{\Quot}{Quot}

\DeclareMathOperator{\Span}{span}
\DeclareMathOperator{\trdeg}{tr.deg}
\DeclareMathOperator{\diag}{diag}

\DeclareMathOperator{\Ga}{{\bf G}_a}
\DeclareMathOperator{\Ru}{{\bf R}_u}

\DeclareMathOperator{\GL}{GL}
\DeclareMathOperator{\Rep}{\bf Rep}
\newcommand{\Gm}{\mathbf{G}_{m}}
\DeclareMathOperator{\Ker}{Ker}
\DeclareMathOperator{\ord}{ord}
\DeclareMathOperator{\id}{id}
\newcommand{\K}{\mathbf{K}}
\newcommand{\ZZ}{\mathbb{Z}}

\newcommand{\Le}{\leqslant}
\newcommand{\Ge}{\geqslant}

\numberwithin{equation}{section}

\begin{document}

\maketitle

\begin{abstract}\footnotesize We deal with aspects of the direct and inverse problems in  parameterized Picard--Vessiot (PPV) theory. It is known that, for certain fields, a linear differential algebraic group (LDAG) $G$ is a PPV Galois group over these fields if and only if $G$ contains a Kolchin-dense finitely generated group. We show that, for a class of LDAGs $G$, including unipotent groups, $G$ is such a group if and only if it has differential type $0$. We give a procedure to determine if a parameterized linear differential equation has a PPV Galois group in this class and show how one can calculate the PPV Galois group of a parameterized linear differential equation if its Galois group has differential type $0$.
\end{abstract}
\smallskip
\keywords{differential algebraic groups; parameterized differential Galois theory; algorithms}
\classification{12H05}{12H20; 13N10; 20G05; 20H20; 34M15}

\section{Introduction}
Classical differential Galois theory studies symmetry
groups of solutions of linear differential equations, or,
equivalently, the groups of automorphisms of the corresponding
extensions of differential fields. The groups that arise are linear
algebraic groups over the field of constants. This theory, started in the 19th century by Picard and Vessiot, was put on a firm
modern footing by Kolchin~\cite{Kolchin1948}. A
generalized differential Galois theory  having differential algebraic groups (as in \cite{KolDAG}) as Galois groups  was initiated in~\cite{Landesman}.
 The parameterized Picard--Vessiot Galois theory considered in~\cite{PhyllisMichael} is a special case of the above generalized differential Galois theory and studies symmetry groups of the solutions of linear differential equations whose coefficients contain parameters. This is done by
constructing a differential field containing the solutions and their
derivatives with respect to the parameters, called a parameterized
Picard--Vessiot extension (PPV-extension), and studying its group of
differential symmetries, called a parameterized differential Galois
group (PPV-group). The Galois groups that arise are linear differential algebraic
groups (LDAGs), which are groups of matrices whose entries satisfy polynomial differential
equations in the parameters. 

As in all Galois theories, one can ask for an answer to the  {\em  Inverse Problem (Which groups appear as Galois groups?)} and  the {\em Direct Problem (Given an equation, what is its Galois group?)}. This paper deals with aspects of both of these problems in the context of the parameterized Picard--Vessiot theory. 

Beginning with the Inverse Problem, let $\U$ be a universal differential field \cite[Ch.~III.7]{Kol} with derivations $\Delta = \{\partial_1, \ldots, \partial_m\}$, that is, a $\Delta$-differential extension of $\C$, the complex numbers, such that, if $k\subset K$ are $\Delta$-fields with $k \subset \U$ and $k, K$ both finitely generated over $\Q$ as $\Delta$-fields, then there is a $\Delta$-isomorphism of $K$ into $\U$ fixing elements of $k$. We extend $\U$ to a $\Delta'= \{\partial\} \cup \Delta$-field $\U(x)$, where the derivations of $\Delta$ extend by setting $\partial_i(x) = 0, i = 1, \ldots , m$, and the new derivation $\partial$ is trivial on $\U$ and $\partial(x) = 1$. Combining the results of ~\cite{DreyfusDensity,ClaudineMichael}, we have the following characterization of those LDAGs that occur as parameterized PPV-groups over $\U(x)$:

\smallskip
\noindent
 {\em A linear differential algebraic group $G$ is a PPV-group over $\U(x)$ if and only if it contains a finitely generated Kolchin-dense subgroup (such an LDAG $G$ is called a differentially finitely generated group (DFGG)).} 

\smallskip 
 \noindent
 One can ask for a characterization of such groups in terms of their group-theoretic structure.  For example, in \cite{MichaelGmGa}, it is shown that a linear algebraic group $G$ (thought of as
an LDAG) has a finitely generated Kolchin-dense subgroup if and only if there is no differential homomorphism of its identity component $G^\circ$  onto $\Ga$, the additive group. In this paper, we prove (Theorem~\ref{thm:Main}) that an LDAG $G$ with $G/\Ru(G)$ constant (see Proposition~\ref{prop:Constant} for the meaning of ``constant''), where $\Ru(G)$ is the unipotent radical of $G$, is a DFGG if and only if $G$ has differential type $0$ (Definition~\ref{Def:type}), that is, $G$ is in a certain sense finite-dimensional. In particular, this characterizes what unipotent groups appear as PPV-groups over $\U(x)$. That is, if $G$ is a unipotent DFGG, then  $G$ has differential type $0$. In the extreme case, when $G\subset\Ga$, see Lemma~\ref{lem:DFGGGa}. 
 
 The difficulty that arises when one attempts to deduce our main result from this fact, by induction, is as follows. If $G_1$ is a normal differential algebraic subgroup  of $G$ such that $G/G_1$ embeds into $\Ga$, then, by the above, the differential type of $G/G_1$ is $0$. Hence, if we knew that $G_1$ had differential type $0$, we would be able to conclude that $G$ had differential type $0$, as desired. However, it is not clear why $G_1$ must be a DFGG, which is one of the subtleties.

Turning to the direct problem, the first known algorithms that compute PPV-groups are given in \cite{Carlos,Dreyfus}. These apply to  first and second orders equations. In this paper, we present two algorithms concerning the direct problem.  In Algorithm 1 (\S\ref{alg1}), we give a procedure that finds the defining equations of the PPV-group of a linear differential equation $\partial_xY= AY,  A\in \Mn_n(\U(x))$ assuming this group has differential type $0$ (Definition~\ref{Def:type}). In Algorithm 2 (\S\ref{alg2}), we give a procedure that determines if the PPV-group $G$ of a linear differential equation $\partial_xY= AY,  A\in \Mn_n(\U(x))$ has the property that $G/\Ru(G)$ is constant. Combining these algorithms allows us to determine if $G/\Ru(G)$ is constant and, if so, find the defining equations of $G$. On the other hand, if $G/\Ru(G)$ is not necessarily constant, an algorithm that computes $G/\Ru(G)$ is given in~\cite{MiOvSiRed}, together with an algorithm that decides whether $G/\Ru(G)=G$.

The paper is organized as follows. In \S\ref{sec:LDAG}, we begin by reviewing some basic facts concerning  differential algebra, differential dimension, linear differential algebraic groups and their representations, and unipotent  differential algebraic groups. We  then show  the result Theorem~\ref{thm:Main} described above. In~\S\ref{sec:PPV}, we review the essential features of the parameterized Picard--Vessiot theory, present the two algorithms described above and give some examples. 

\section{Linear Differential Algebraic Groups}\label{sec:LDAG}
\subsection{Differential Algebra}\label{subsec:DA}
 We recall some definitions and facts from differential algebra. General references for this section are \cite{Kol} and \cite{Kap}. A $\Delta=\{\partial_1,\ldots,\partial_m\}$-ring $R$ is a commutative associative ring with unit $1$ and commuting derivations $\partial_i: R\to R$. 
Let
$$
\Theta = \big\{\partial_1^{i_1}\cdot\ldots\cdot\partial_m^{i_m}\:|\: i_j \Ge 0\big\}\quad\text{and}\quad \ord\big(\partial_1^{i_1}\cdot\ldots\cdot\partial_m^{i_m}\big) := i_1+\ldots+i_m.$$
Since $\partial_i$ acts on $R$,
there is a natural action of $\Theta$ on $R$.
Let $R$ be a $\Delta$-ring. If $B \supset R$, then $B$ is a $\Delta$-$R$-algebra
if the action of $\Delta$ on $B$ extends the
action of $\partial$ on $R$.
Let $Y = \{y_1,\ldots,y_n\}$ be a set of variables and
$$
\Theta Y := \left\{\theta y_j
\:\big|\: \theta\in\Theta,\ 1\Le j\Le n\right\}.
$$
The ring of differential polynomials $R\{Y\}$ in
differential indeterminates $Y$
over $R$ is
$R[\Theta Y]$
 with
the derivations $\partial_i$ that 
extend the $\partial_i$-action on $R$ as follows:
$$
\partial_i\left(\theta y_j\right) := (\partial\cdot\theta)y_j,\quad 1 \Le j \Le n.$$
 An ideal $I$ in a $\Delta$-ring $R$ is called a differential ideal if
$
\partial_i(a) \in I$ for all  $a \in I$, $1\Le i\Le m$. 
For $F \subset R$,  $[F]$ denotes the differential ideal of $R$ generated by $F$.

 Let $\U$ be a  differentially closed $\Delta$-field, that is a $\Delta$-field  such that any system of polynomial differential equations with coefficients in $\U^n$ having a solution in some $\Delta$-extension of $\U$ already have a solution in $\U^n$ (see \cite[Definition~3.2]{PhyllisMichael}, \cite[Definition~4]{TrushinSplitting}, and the references given there; we do not assume that $\U$ is universal)
and let  $\Const\subset\U$ be its subfield of constants, that is, $\Const = \bigcap\ker\partial_i$.

\begin{definition} A {\it Kolchin-closed} subset $W$ of $\U^n$ defined over $k$ is the set of common zeroes
of a system of differential algebraic equations with coefficients in $k,$ that is, for $f_1,\ldots,f_r \in k\{Y\}$, we define
$$
W = \left\{ a \in \U^n\:|\: f_1(a)=\ldots=f_r(a) = 0\right\}.$$
\end{definition}
\noindent A differential version of the usual Nullstellensatz implies that there is  a bijective correspondence between Kolchin-closed subsets
$W$ of $\U^n$ defined over $k$ and radical differential ideals
$\I(W) \subset k\{y_1,\ldots,y_n\}$. In particular, $\I(W)$ is the smallest radical differential ideal containing
 the differential polynomials $f_1,\ldots,f_r$ that define $W$. The Ritt--Raudenbush Basis Theorem states that every radical
differential ideal is the radical of a finitely
generated differential ideal and so $k\{Y\}$ satisfies the ascending chain condition on radical differential ideals.
Given a Kolchin-closed subset $W$ of
$\U^n$ defined over $k$, we let the coordinate ring $k\{W\}$ be defined as

$$
k\{W\} = k\{y_1,\ldots,y_n\}\big/\I(W).
$$
A differential polynomial map $\varphi : W_1\to W_2$ between Kolchin-closed subsets of $\U^{n_1}$ and $\U^{n_2}$, respectively, defined over $k$, is given in coordinates by differential polynomials in
$k\{W_1\}$. Moreover, to give $\varphi : W_1 \to W_2$
is equivalent to defining a differential $k$-homomorphism $\varphi^* : k\{W_2\} \to k\{W_1\}$. If $k\{W\}$ is an integral domain, then $W$ is called {\it irreducible}. This is equivalent to $\I(W)$ being a prime differential ideal. More generally, if $$\I(W) = \p_1\cap\ldots\cap \p_q$$ is a minimal prime decomposition, which is unique up to permutation, \cite[VII.29]{Kap}, then the irreducible Kolchin-closed sets 
$W_1,\ldots, W_q$ corresponding to $\p_1,\ldots,\p_q$ are called the {\it irreducible components} of $W$. We then have $$W = W_1\cup\ldots\cup W_q.$$

We will need to measure the ``size'' of Kolchin-closed sets. In algebraic geometry, the Hilbert function allows us to  define the dimension and degree of certain algebraic sets using the Hilbert  function, and  a similar object, called the Kolchin polynomial, is used to measure properties of Kolchin-closed set.  Given a differential field extension $k\subset L$ and $a=(a_1,\ldots, a_n )\in L^n$, let $M:=k( a_1,\ldots,a_n)$
denote the  subfield of $L$ generated by $a_1,\ldots,a_n$ over $k$ (in the algebraic, not differential sense). Let 
$$M_s =  k(\theta (a_i)\:|\: \ord(\theta)\Le s, 1\Le i\Le n).$$ One can show (\cite[\S II.12]{Kol}) that there exist  a polynomial $\omega_{a/k}(t) \in \Q(t)$  and an integer $N$ such that 
$$\omega_{a/k}(s) = \trdeg_KM_s,\quad \text{ for all } s>N.$$
\begin{definition} \label{Def:type}\hspace{0cm}
\begin{enumerate}
\item The polynomial $\omega_{a/k}(s)$ is called the {\it Kolchin polynomial of $a$ over $k$}.  
\item The degree of $\omega_{a/k}(t)$ is called the {\it differential type of a over $k$} and is denoted by $\tau(a/k)$ (if $\omega_{a/k}(t) = 0$, we let $\tau(a/k) = -\infty$). 
\item  If $W$ is an irreducible Kolchin-closed set over $k$, $\omega(W)$ and $\tau(W)$ are defined to be $\omega_{a/k}$ and $\tau(a/k)$, respectively, where $a$ is the image of $(y_1, \ldots , y_n)$ in $$k\{W\} = k \{y_1,\ldots,y_n\}\big/\I(W).$$ If $V$ is an arbitrary Kolchin-closed set over $k$, we define
\begin{align*}
\tau(V) &= \max \{\tau(W) \ | \ W \text{ an irreducible component of } V\},\\
\omega(V) &= \max \{\omega(W) \ | \ W \text{ an irreducible component of } V\}.
\end{align*}\end{enumerate}
\end{definition}
We refer to \cite[\S II.12]{Kol}  and \cite[\S IV.4]{KolDAG} for the the various properties of the differential type. The following are the two properties of differential type that are most important for this paper. For an irreducible Kolchin-closed set $V$ over $k$, we denote  the quotient field of $k\{V\}$ by $k\langle V\rangle$.
\begin{enumerate}
\item \label{tau1}If $V$ is an irreducible  Kolchin-closed set over $k$, then $\tau(V) = 0$ if and only if 
$\trdeg_k k\langle V\rangle < \infty$.
\item \label{tau2} If $H \subset G$ are linear differential algebraic groups (see Section~\ref{Sec:LDAG}), then $\tau(G) = \max \{\tau(G/H), \tau(H)\}$.
\end{enumerate} 
Property~\ref{tau1} follows immediately from the definition of type and Property~\ref{tau2}  is proven in \cite[\S IV.4]{KolDAG}.

\subsection{Linear Differential Algebraic Groups: Structure and Representations}\label{Sec:LDAG}

In this section, we review the general facts concerning linear differential algebraic groups that we need in the succeeding sections.

\begin{definition}\cite[Ch.~II, \S1, p.~905]{Cassidy}\label{def:LDAG} A {\it linear differential
algebraic group} (LDAG) defined over $k$ 
is a Kolchin-closed subgroup $G$  of $\GL_n(\U)$ over $k$,
that is, an intersection
of a Kolchin-closed subset of $\U^{n^2}$ over $k$ with $\GL_n(\U)$ that is closed under
the group operations.
\end{definition}

Although several of the results mentioned in this section hold for LDAGs defined over arbitrary $k$, we shall now assume, without further mention,  that $k = \U$, that is, the LDAGs we deal with are defined over a differentially closed field.

Note that we identify $\GL_n(\U)$ with the Kolchin-closed
subset of $\U^{n^2+1}$ given by
$$\left\{(A,a)\:\big|\: (\det(A))\cdot a-1=0\right\}.$$
If $X$ is an invertible $n\times n$ matrix, we
can identify it with the pair $(X,1/\det(X))$. Hence, we may represent the coordinate ring of $\GL_n(\U)$ as
$
\U\{X,1/\det(X)\}$.
Denote $\GL_1$ simply by $\Gm$. Its coordinate ring
is $\U\{y,1/y\}$.
The group $(\U, +)$ may be considered an LDAG via its usual two-dimensional representation. We shall refer to this LDAG as $\Ga$.  Its coordinate ring is $\U\{y\}$.
 
For a group $G\subset\GL_n(\U)$, we denote  the Zariski and Kolchin-closures of $G$ in $\GL_n(\U)$ by $\overline{G}^Z$ and $\overline{G}$, respectively. Note that $\overline{G}^Z$  and $\overline{G}$ are a linear algebraic group and LDAG over $\U$, respectively. 

The irreducible component of an LDAG $G$ containing the element $\id$ is called the {\it identity component} of $G$ and denoted by $G^\circ$. An LDAG $G$ is called {\it connected} if $G = G^\circ$, which is equivalent to $G$ being an irreducible Kolchin-closed set \cite[p.~906]{Cassidy}.

\subsubsection{Representations of linear differential algebraic groups.}\label{sec:rep}
\begin{definition}\cite{CassidyRep},\,\cite[Definition~6]{OvchRecoverGroup} Let $G$ be an LDAG. A differential polynomial
group homomorphism  $$\rho : G \to \GL(V)$$ is called a
{\it differential representation} of $G$, where $V$ is a
finite-dimensional vector space over $k$. Such a space is
simply called a {\it $G$-module}. 
{\it Morphisms} between $G$-modules are $k$-linear maps that are $G$-equivariant. The category of differential representations of $G$ is denoted by $\Rep G$.
\end{definition}

By \cite[Proposition~7]{Cassidy}, $\rho(G)\subset\GL(V)$ is a differential algebraic subgroup.  
Given a representation $\rho$ of an LDAG $G$, for each $i$, $1\Le i\Le m$, one can define its prolongations 
$$P_i(\rho) : G \to \GL(P_i(V))
$$ with respect to $\partial_i$ as follows (see \cite[\S5.2]{GGO}, \cite[Definition~4 and Theorem~1]{OvchRecoverGroup}, and \cite[p.~1199]{diffreductive}). Let 
\begin{equation}\label{eq:prolongation}
P_i(V):=\leftidx{_{\U}}{\left((\U\oplus \U\partial_i)_{\U}\otimes_{\U} V\right)}
\end{equation}
as vector spaces, where $\U\oplus \U\partial_i$ is considered as the right $\U$-module:
$$
\partial_i\cdot a = \partial_i(a) + a\partial_i
$$
for all $a \in \U$.
Then the action of $G$ is given by $P_i(\rho)$ as follows:
$$
P_i(\rho)(g) (1\otimes v) := 1\otimes \rho(g)(v),\quad P_i(\rho)(g)(\partial_i\otimes v) := \partial_i\otimes\rho(g)(v)
$$
for all $g \in G$ and $v \in V$. In the language of matrices, if $A_g \in \GL_n$ corresponds to the action of $g \in G$  on $V$, then the matrix
$$
\begin{pmatrix}
A_g&\partial_i\big(A_g\big)\\
0&A_g
\end{pmatrix}
$$
corresponds to the action of $g$ on $P_i(V)$. In what follows, the $q^{\rm th}$ iterate of $P_i$ is denoted
by $P_i^q$. For any integer $s$ and LDAG $G \subset \GL_n(\U)$, we will refer to $$P_m^sP_{m-1}^s\ldots P_1^s: G \rightarrow GL_{N_s}(\U)$$  to be the {\em $s^{th}$ total prolongation of $G$} (where $N_s$ is the dimension of the underlying prolonged vector space). We denote this representation by $P^s:G \rightarrow GL_{N_s}(\U)$.

\subsubsection{Unipotent radical of  differential algebraic groups.}~Analogous to linear algebraic groups, one can define the notion of a unipotent LDAG and the unipotent radical of an LDAG.

\begin{definition}\label{def:unipotent}\cite[Theorem~2]{Cassunipot} Let $G \subset \GL_n(\U)$ be a linear differential algebraic group defined over $k$. We say that $G$ is {\it unipotent} if one of 
the following equivalent conditions holds:
\begin{enumerate}
\item $G$ is conjugate to a differential algebraic subgroup  of the group $\mathbf{U}_n$ of unipotent upper triangular matrices.
\item $G$ contains no elements of finite order greater than $1$.
\item $G$ has a descending normal sequence  of differential algebraic subgroups 
$$G=G_0 \supset G_1 \supset \ldots \supset G_N =\{1\}.$$
with $G_i/G_{i+1}$ isomorphic to a differential algebraic subgroup of the additive group $\bold{G}_a$.
\end{enumerate}
\end{definition}

\begin{example} Typical examples of unipotent LDAGs include $\Ga$, its powers, and their differential algebraic subgroups \cite[Proposition~11]{Cassidy} as well as the $\mathbf{U}_n$ and their differential algebraic subgroups.
\end{example}

The image of a unipotent LDAG under a differential polynomial homomorphism is again unipotent \cite[Proposition~35]{Cassunipot} and, therefore,  a unipotent LDAG $G$ is connected (since the image of $G$ in $G/G^\circ$ is a finite unipotent group and, therefore, trivial). One can also show that 
a linear differential algebraic group  $G$  admits a unique maximal normal unipotent differential algebraic subgroup $\Ru(G)$  \cite[Theorem~3.10]{diffreductive}.

\begin{definition}
Let $G$ be an LDAG. Then
\begin{enumerate} \item $\Ru(G)$ is  called the {\it unipotent radical} of $G$; 
\item $G$ is called {\it reductive} if $\Ru(G) = \{\id\}$. \end{enumerate}
\end{definition}

\begin{example}
Typical examples of reductive LDAGs include $\GL_n$ and its Zariski dense differential algebraic subgroups \cite{CassidyClassification}, differential algebraic subgroups of $\Gm$, the LDAG $\SL_n$, and other LDAGs whose defining polynomial equations define semi-simple linear algebraic groups. 
\end{example}

\begin{remark}\label{rem:GbarG} Note that
\begin{enumerate}
\item Reductivity of an LDAG does not depend on its faithful representation.
\item If $G$ is given as a linear differential algebraic subgroup of some $\GL_n$, we may consider its Zariski closure $H$
in $\GL_n$, which is a linear algebraic group defined over $\U$. We will denote the unipotent radical of $H$ (in the sense of linear algebraic groups) by $\Ru(H)$ as well. Following the proof of \cite[Theorem~3.10]{diffreductive}, one then has $$\Ru(G) = \Ru{\left(H\right)} \cap G.$$ 
This implies that, if $H$ is reductive as a linear algebraic group, then $G$ is reductive.
However, the Zariski closure of $\Ru(G) $ may be strictly 
included in $\Ru(H)$ \cite[Example~3.17]{diffreductive}. 
\end{enumerate} 
\end{remark}

\subsubsection{Constant Linear Differential Algebraic Groups.} Recall that an additive category is called {\it semisimple} if, for every object $V$ and subobject
$W \subset V$, there exists a subobject $U\subset V$ such that $V = W \oplus U$.  In charactersitic zero, the
category of finite-dimensional $G$-modules  of a reductive algebraic group $G$ is semisimple \cite[Ch.~2]{Springer}.

\begin{definition}We will call an LDAG $G$ over $\U$ \emph{constant} if the category $\Rep G$ is semisimple.\end{definition}
\begin{proposition}\cite[Theorem~3.14]{diffreductive},\cite[Theorem~4.6]{GO}\label{prop:Constant}
Let $G$ be a reductive LDAG. Let $\varrho: G\to\GL_n(\U)$ be a faithful differential  representation. The following statements are equivalent:
\begin{enumerate}
\item[(1)] $G$ is constant;
\item[(2)] $\varrho(G)$ is conjugate to a subgroup of $\GL_n(\Const)$ by a matrix from $\GL_n(\U)$;
\item[(3)] for the $G$-module $V=\U^n$, where $G$ acts via $\varrho$, $P_i(V)\simeq V\oplus V$ for all $i$, $1\Le i \Le m$.
\end{enumerate}
\end{proposition}

Simply put, a reductive LDAG $G$ is constant if  it is differentially isomorphic to a group of constant matrices. In this case, images of all representations of $G$ are conjugate to groups of constant matrices.

\subsection{Differentially finitely generated groups}\label{sec:DFGG}
As mentioned in the introduction, this paper is motivated by a desire to further understand  LDAGs that have finitely generated Kolchin-dense subgroups. 

\begin{definition}\label{def:DFGG}
An LDAG $G$ defined over $\U$ is said to be \emph{differentially finitely generated}, or simply a \emph{DFGG}, if $G(\U)$ has a Kolchin-dense finitely generated subgroup. \end{definition}

A general question is to characterize such groups in terms of their algebraic structure. For example, in \cite{MichaelGmGa}, it is shown that a linear algebraic group $G$ (thought of as an LDAG) is a DFGG if and only if there is no differential homomorphism of its identity component $G^\circ$ onto $\Ga$. In this section, we shall characterize another class of differentially finitely generated groups. In particular, we will show
\begin{theorem}\label{thm:Main}Let  $G$ be an LDAG  with $G/\Ru(G)$ constant. The group $G$ is a DFGG if and only if $\tau(G) \Le 0$.
\end{theorem}
We begin by showing that the condition  $\tau(G) \Le 0$ implies that $G$ is a DFGG.
\begin{proposition}\label{prop:sufficiency}  Let $G$ be an LDAG defined over $\U$. If $\tau(G) \Le 0$, then $G$ is a DFGG.
\end{proposition}
\begin{proof} The proof follows, {\it mutatis mutandis}, the proof given in \cite{tretkoff79} that a linear algebraic group contains a finitely generated Zariski-dense subgroup.  For the convenience of the reader, we supply the details. We first prove this result under the additional assumption that $G$ is connected. Note that $\tau(G) \Le 0$ implies that $\omega(G)$ is an integer equal to the transcendence degree of the quotient field of $\U\{G\}$ over $\U$.   When $$\omega(G) = 0,$$ then $G$ is finite and the result is clear.  
Assume now that $$\omega(G) > 0.$$
We will now show that $G$ contains an element of infinite order.  Assume not, that is all elements of $G$ have finite order. A theorem of Schur \cite[Theorem~36.14]{curtisreiner} states that a linear group all of whose elements are of finite order has an abelian normal subgroup of finite index. Therefore, since $G$ is connected, it must be abelian. Furthermore, we can assume that, after a conjugation, all elements are diagonal, that is, $$G \subset \Gm^t$$ for some $t$. Since $G$ is connected and infinite, its Zariski closure $\overline{G}^Z$ is isomorphic to $\Gm^q$ for some $q>0$. \cite[Proposition~31]{Cassidy} implies that $G$ contains $\Gm^q(\Const)$. This latter group clearly contains elements of infinite order, so $G$ contains an element of infinite order.

We now proceed by induction on $\omega(G)$.  First note that, if $V\subset W$ are  
irreducible Kolchin-closed sets, then $$\omega(V) (s) \Le \omega(W)(s)$$ for all sufficiently large $s$, and equality holds if and only if $V=W$ \cite[Proposition~III.5.2]{Kol}. Let $g \in G$ be an element of infinite order. We then have $$1 \Le \omega\big(\overline{\langle g \rangle}\big) \Le \omega(G),$$ where $\overline{\langle g \rangle}$ is the Kolchin-closure of the group generated by $g$.  If $\omega(G) = 1$, then $\overline{\langle g \rangle} = G$. If $\omega(G) > 1$, let $H$ be a maximal proper connected subgroup of $G$, which exists, for instance, because $\tau(G)\Le 0$. We have that $$1 \Le \omega(H) < \omega(G),$$ so $H$ is a DFGG. 

 If $H$ is a normal subgroup of $G$, then $G/H$ may be identified with a connected LDAG and $$\omega(G/H) = \omega(G) - \omega(H) < \omega(G)$$ (\cite[Corollary~3]{KolDAG}; note that, since these groups have differential type $0$, $a_{\Delta, \mu}(\ldots) = \omega(\ldots)$ in Kolchin's result). Using the inductive hypothesis, let $\{h_i\}$ be a finite set generating a Kolchin-dense subgroup of $H$ and $\{g_j\} \subset G$ be a finite set whose images generate a Kolchin-dense subgroup of $G/H$. The set $\{g_j,h_i\}$ generates a Kolchin-dense subgroup of $G$.
 
 If $H$ is not normal in $G$, then there exists $g \in G$ such that $gHg^{-1} \neq H$. Let $$H = \overline{\langle h_1, \ldots , h_t\rangle}\quad \text{and}\quad L =  \overline{\langle h_1, \ldots , h_t, g\rangle}.$$ The identity component $L^\circ$ of $L$ contains both $H$ and $gHg^{-1}$ and so properly contains $H$. Therefore, $L^\circ = G$, and so $L = G$. This completes the proof when $G$ is connected. For a general LDAG $G$, we need only to note that $G$ is generated by $G^\circ$ and representatives from each of the finite number of cosets of $G^\circ$.  \end{proof}
 
 We will show  in Proposition~\ref{prop:type0} that, if $G$ is a DFGG with $G/\Ru(G)$ constant, then $\tau(G) \Le 0$.   The proof depends on the following four lemmas.

\begin{lemma}\label{lem:DFGGGa} If $H\subset \Ga$ is a DFGG, then $\tau(H) \Le 0$.
\end{lemma}
\begin{proof}
Let $h_1, \ldots ,h_m$ be nonzero elements that generate a Kolchin-dense subgroup of $H$.  We shall show that, for each $i$, $1\Le i\Le n$, there is a nonzero linear differential operator $L_i \in \U[\partial_i]$ such that $L_i(h) = 0$ for all $h \in H$.  This implies that the  (usual, not differential) transcendence degree of the differential coordinate ring of $H$ is finite and so $H$ has type $0$.
Let $C_i = \U^{\partial_i}$ be the $\partial_i$-constants of $\U$. Fix  $i$ and let $w_1, \ldots , w_t$ be a $C_i$-basis of $\Span_{C_i}\{h_1, \ldots , h_m\}$. Let 
$$L_i(Y) = \det \left(\begin{array}{cccc} Y&\partial_iY&\ldots &\partial_i^tY\\w_1 &\partial_iw_1&\ldots &\partial_i^tw_1\\ \vdots&\vdots&\vdots&\vdots\\w_t &\partial_iw_t&\ldots &\partial_i^tw_t\end{array}\right),$$
that is, $$L_i(Y) = wr(Y, w_1, \ldots , w_t),$$ where $wr$ is the Wronskian determinant with respect to $\partial_i$. All linear combinations of the $h_j$ satisfy $L_i(Y) = 0$, so $L_i(h) = 0$ for all $h \in H$.
\end{proof}

For a group $G$ and its subgroups $G_1$ and $G_2$, we write
$$
G=G_1G_2,
$$
if every element $g\in G$ can be represented as a product of elements $g_1\in G_1$ and $g_2\in G_2$. Moreover, if $G_2$ is normal and $G_1\cap G_2=\{e\}$, we write
$$
G=G_1\ltimes G_2.
$$

\begin{proposition}\label{prop:SDirect}
Let $\Gamma_i\subset\GL_n$, $i=1,2,3$, be subgroups and
$$
\Gamma_3=\Gamma_1\ltimes\Gamma_2. 
$$
Let $G_i\subset\GL_n$ be the Kolchin-closure of $\Gamma_i$, $i=1,2,3$. If $G_1\cap G_2=\{e\}$, then
$$
G_3=G_1\ltimes G_2.
$$
\end{proposition}

\begin{proof}
By the argument as in \cite[Theorem~4.3(b)]{Waterhouse}, $G_2$ is normal in $G_3$. Consider the quotient map
$$
\pi: G_3\to Q:=G_3/G_2.
$$
Since $\Gamma_3$ is dense in $G_3$, $\pi(\Gamma_3)=\pi(\Gamma_1)$ is dense in $Q$. Therefore, $\pi(G_1)=Q$. Hence, $G_3=G_1G_2$. The hypothesis $G_1\cap G_2=\{e\}$ finishes the proof.
\end{proof}

For a subset $X\subset\U$, let $A(X)$ denote the differential algebraic subgroup of $\Ga(\U)$ generated by $X$. 
\begin{lemma}\label{prop:ConstSpan}
Let $X\subset\U$. Then $A(X)$ contains the $\Const$-span of $X$.
\end{lemma}

\begin{proof}
By \cite[Proposition~11]{Cassidy}, $A(X)$ is given by a system of linear differential equations. Hence, $A(X)$ is vector space over $\Const$.
\end{proof}

For a subset $X\subset\U$ and $n\Ge 1$, let $A_n(X)\subset\GL_n(\U)$ be the set (not a group!) of upper-triangular unipotent matrices whose entries above the diagonal are from $A(X)$. Note that $A_n(X)$ is Kolchin-closed in $\GL_n(\U)$. Let $I_n\in\GL_n(\U)$ denote the identity matrix  and $I_{ij}$, $1\Le i,j\Le n$, denote the matrix whose elements are all $0$ except for the element in the $i^{th}$ row and $j^{th}$ column, which is $1$. Set $U_n(X)\subset\GL_n(\U)$ to be the Kolchin-closure of the group $\Gamma_n(X)$ generated by  $$\{I_n+x\cdot I_{ij}\:|\: x\in X,\, i<j\}.$$

\begin{lemma}\label{lem:UnType0}
Let $X$ be a subset of $\U$. Then $A_n(X)\subset U_n(X)$ and, if $X$ is finite, $\tau(U_n(X))\Le 0$.
\end{lemma}

\begin{proof}
We will use induction on $n$, with the case $n=1$ being trivial. Note that
\begin{equation}\label{eq:matprod}
\begin{pmatrix}
A & u\\
0 & 1
\end{pmatrix}=
\begin{pmatrix}
I & u\\
0 & 1
\end{pmatrix}
\begin{pmatrix}
A & 0\\
0 & 1
\end{pmatrix},
\end{equation}
for all $A\in\GL_{n-1}$, $u\in\Ga^{n-1}$, where $I$ is the identity matrix. Since, for all $u\in A(X)^{n-1}$,
$$
\begin{pmatrix}
I & u\\
0 & 1
\end{pmatrix}\in U_n(X),
$$
\eqref{eq:matprod} and the induction imply that
$$
A_n(X)\subset U_n(X).
$$
Set $$X_1:=X\cup\{1\}\subset\U.$$ For an integer $r> 0$, let $X_1^r\subset\U$ denote the set of all products of $r$ elements from $X_1$. In particular, $X\subset X_1^r$. The subset
$$
B_n(X):=\big\{B=(b_{ij})\in U_n(\U)\:|\: b_{ij}\in A\big(X_1^{j-i}\big)\big\}\subset \GL_n(\U)
$$
is Kolchin-closed, since so are $A\big(X_1^r\big)\subset\U$, $r>0$. Moreover, $B_n(X)$ is a subgroup of $\GL_n(\U)$. Indeed, for a subset $Y\subset\U$, let $\Gamma(Y)\subset\U$ stand for the additive group generated by $Y$. Then $A(Y)$ is the Kolchin closure of $\Gamma(Y)$. The product map
$$
\mu:\U\times\U\to\U,\ (x,y)\mapsto xy,
$$
sends $X_1^r\times X_1^s$ to $X_1^{r+s}$, whence, by bilinearity of the product,
$$
\mu\big(\Gamma\big(X_1^r\big)\times \Gamma\big(X_1^s\big)\big)\subset \Gamma\big(X_1^{r+s}\big)
$$
for all integers $r,s>0$. Since, for arbitrary subsets $Y,Z\subset\U$, the Kolchin closure of $Y\times Z\subset\U\times\U$ equals the (set) product of the Kolchin closures of $Y$ and $Z$, and $\mu$, being an algebraic map, sends Kolchin closures to Kolchin closures,
$$
A\big(X_1^r\big) A\big(X_1^s\big)=\mu\big(A\big(X_1^r\big)\times A\big(X_1^s\big)\big)\subset A\big(X_1^{r+s}\big).
$$
This implies that $B_n(X)$ is a (differential algebraic) subgroup of $\GL_n(\U)$. Since $B_n(X)$ contains every matrix from $U_n(\U)$ whose entries above the diagonal belong to $X$, we conclude that
\begin{equation}\label{eq:finemb}
U_n(X)\subset B_n(X).
\end{equation}
Suppose that $X$ is finite. Then, for every integer $r>0$, $X_1^r$ is finite. By Lemma~\ref{lem:DFGGGa},  $$\tau\big(A\big(X_1^r\big)\big)\Le 0.$$ The definition of $B_n(X)$ implies that $\tau(B_n(X))\Le 0$. Finally, by~\eqref{eq:finemb}, $\tau(U_n(X))\Le 0$.
\end{proof}

\begin{proposition}\label{prop:type0}
Let $G\subset\GL(V)$ be a DFGG such that $G/\Ru(G)$ is constant. Then $\tau(G)\Le 0$.
\end{proposition}
\begin{proof}
Fix a flag of maximal length 
$$
\{0\}=V_0\subset V_1\subset\ldots\subset V_n=V
$$
of submodules of the $G$-module $V$. Then all the quotients $V_i/V_{i-1}$, $1\Le i\Le n$, are simple. For each $V_i$, $1\Le i \Le n$, let $W_i$ be a complementary subspace to $V_{i-1}$:
$$
V_i=V_{i-1}\oplus W_i\quad\text{(as vector spaces)}.
$$
We have
$$
V=\bigoplus_{i=1}^nW_i\quad\text{(as vector spaces)}.
$$
Consider the following algebraic subgroups of $\GL(V)$:
\begin{align*}
& P = \{g\in\GL(V)\:|\:gV_i\subset V_i,\: 1\Le i\Le n\},\\
& U = \{g\in\GL(V)\:|\:gv-v\in V_{i-1}\ \text{for all } v\in V_i, 1\Le i\Le n\},\\
& R = \{g\in\GL(V)\:|\:gW_i\subset W_i,\: 1\Le i\Le n\}.
\end{align*}
By definition, $U \subset P$ and $R \subset P$. Moreover, for all $h \in P$, $g \in U$, $i$, $1\Le i \Le n$, and $v \in V_i$, we have
$$
hgh^{-1}v - v = hgh^{-1}v - hh^{-1}v = h\big(gh^{-1}v - h^{-1}v\big) \in V_{i-1}.
$$
Therefore, $U$ is normal in $P$. Let now $$g \in U\cap R\quad \text{and}\quad v = w_1+\ldots+w_n,\ \  w_i \in W_i,\ \ 1\Le i \Le n.$$  We then have:
$$
gw_i -w_i \in W_i\cap V_{i-1} = \{0\},
$$
which implies that $g(v) = v$, that is, $g = e$. We will now show that $$P = RU.$$ For this, let $g \in P$. We will construct $t_g \in R$ such that $$t_g^{-1}g \in U$$ by induction on $i$.  By this, we mean that we will suppose  that, for $i \Ge 0$, $t_g$ is defined on $V_i$. We will then define it on $V_{i+1}$ extending its action on $V_i$. Let $v \in V_{i+1}$. There then  exist unique $u \in V_i$ and $w \in W_{i+1}$ such that $$v = u+w.$$ There also exist unique $u'\in V_i$ and $w' \in W_{i+1}$ such that $$g(w) = u'+w'.$$ 
Moreover, since $$g|_{V_i} : V_i \to V_i\quad \text{and}\quad g|_{V_{i-1}}: V_{i-1}\to V_{i-1}$$ are  isomorphisms  of vector spaces and $V_{i-1} \subsetneq V_i$, if $w\ne 0$, then $w'\ne 0$. We let $$t_g(v) := t_g(u)+w',$$
which is well-defined and invertible by the above.
In addition, if $z \in V_{i+1}$, $x,\,x' \in V_i$, and $y,\,y'\in W_{i+1}$ are such that
$$
z = x+y\quad\text{and}\quad g(y)=x'+y',
$$
then, for all $a,\,b \in \K$, by induction, we have
$$
t_g(av+bz)=t_g(au+bx)+aw'+by'=at_g(u)+bt_g(x)+aw'+by'=at_g(v)+bt_g(z),
$$
as $aw'+by' \in W_{i+1}$.
Therefore, $t_g$ is $\K$-linear. 
We will now show that, for  $v \in V_{i+1}$, 
\begin{equation}\label{eq:tg}
t_g^{-1}g(v)-v \in V_{i},
\end{equation}
Since $t_g$ is invertible and preserves $V_{i}$,~\eqref{eq:tg} is equivalent to $$g(v)-t_g(v) \in V_{i}.$$ This is true since 
$$ g(v)-t_g(v) = g(u)+g(w)-t_g(u)-w' = g(u)-t_g(u)+u'.$$
Thus,
$$
P=R\ltimes U.
$$
Since $V_i/V_{i-1}$ are all simple, the images of $G$ in $\GL(V_i/V_{i-1})$ are reductive (see, e.~g., \cite[pf. of Theorem~4.7]{diffreductive}). Therefore, $$\Ru(G)\subset U.$$ Since $G\cap U$ is a normal unipotent subgroup of $G$, we conclude
\begin{equation}\label{eq:radGsubU}
\Ru(G)=G\cap U.
\end{equation}
Let $\Gamma\subset G$ be a Kolchin-dense subgroup generated by a finite set $S\subset G$. Since $G\subset P$, for every $s\in S$, there are $r_s\in R$ and $u_s\in U$ such that $s=r_su_s$. Let $\tilde{\Gamma}\subset P$ be the subgroup generated by $\{r_s, u_s\:|\: s\in S\}$ and $\tilde{G}$ be the Kolchin-closure of $\tilde{\Gamma}$ in $P$. By construction, 
$\Gamma\subset\tilde{\Gamma}$, hence, 
\begin{equation}\label{eq:sub}
G\subset\tilde{G}.
\end{equation}
Note that $$\tilde{\Gamma}=\Gamma_r\ltimes\Gamma_u,$$ where $\Gamma_r$ and $\Gamma_u$ are the subgroups of $\tilde{\Gamma}$ generated by $$\big\{r_s\:|\:s\in S\big\}\qquad\text{and}\qquad\big\{ru_sr^{-1}\:\big|\:s\in S,\,r\in\Gamma_r\big\},$$ respectively. Let $G_r\subset R$ and $G_u\subset U$ be the Kolchin closures of $\Gamma_r$ and $\Gamma_u$, respectively. Then, by Proposition~\ref{prop:SDirect},
\begin{equation}\label{eq:semidirect}
\tilde{G}=G_r\ltimes G_u.
\end{equation}
Let us show that
$$
\Ru(G)\subset G_u\qquad\text{and}\qquad G_r\simeq G/\Ru(G).
$$
By~\eqref{eq:sub},
$$
\Ru(G)\subset\tilde{G}\cap U=G_u.
$$
Now, the projection 
$$
\gamma : P=R\ltimes U \to R
$$
maps $\Gamma$ onto $\Gamma_r$ by construction. Hence, $\gamma(G)=G_r$. Since, by~\eqref{eq:radGsubU},
$$
G\cap\Ker\gamma=G\cap U=\Ru(G),
$$
we conclude $G_r\simeq G/\Ru(G)$.

Since $G_r$ is constant, all its quotients are constant. Therefore, by Proposition~\ref{prop:Constant}, for every $i$, $1\Le i\Le n$, one can choose a basis $B_i\subset W_i$ whose $\Const$-span is $G_r$-invariant. Identifying $\GL(V)$ with $\GL_n(\U)$ using the basis $\bigcup_iB_i$ of $V$, we see that $$G_r\subset\GL_n(\Const)\quad \text{and}\quad U\subset U_n(\U).$$ Let $X\subset\U$ be the (finite) set of matrix entries of all $u_s$, $s\in S$. Note that, if $c\in\GL_n(\Const)$, the matrix entries of $cu_sc^{-1}$ belong to the $\Const$-span of $X$. Hence, by Lemma~\ref{prop:ConstSpan}, for all $r\in\Gamma_r$ and $s\in S$,
$$
ru_sr^{-1}\in A_n(X).
$$
Hence, by Lemma~\ref{lem:UnType0},
$$
\Gamma_u\subset U_n(X).
$$
By the same Lemma, since $G_u\subset U_n(X)$, $\tau(G_u)\Le 0$. We conclude that $$\tau(G)\Le \tau(\tilde{G})=\max\{\tau(G_r),\tau(G_u)\}\Le 0.\qedhere$$
\end{proof}

 We will now give an example of an LDAG $G$ that is DFGG but with $\tau(\Ru(G))=1$. In particular, it shows that the statement of Theorem~\ref{thm:Main} becomes false if one removes the condition that $G/\Ru(G)$ be constant. As is the case for the example given in~\cite[p.~159]{MichaelGmGa}, it also demonstrates that \cite[Proposition~1.2]{MichaelGmGa} does not directly generalize from linear algebraic groups to LDAGs.

\begin{example}(see also \cite[Example~7.2]{PhyllisMichael}) Let $\Delta=\{\partial\}$ and 
$$G = \left\{g(a,b) := \begin{pmatrix} a &b\\0&1\end{pmatrix}\:\Big|\:  a\in \U^*,b \in \U, \partial(\partial(a)/a)=0\right\}   \subset \GL_2.$$ 
Note that $G = G_1\ltimes G_2$, where 
$$G_1 = \left\{g(a,0) := \begin{pmatrix} a &0\\0&1\end{pmatrix}\:\Big|\: a\in \U^*, \partial(\partial(a)/a)=0\right\}$$ and $$G_2 = \left\{g(1,b) := \begin{pmatrix} 1 &b\\0&1\end{pmatrix}\:\Big|\: b \in \U\right\}.$$ 
Let $a\in \U$ be such that $\partial(a) \ne 0$ and $\partial(\partial(a)/a)=0$. We claim that $\{g(a,0), g(1,1)\}$ generates a Kolchin-dense subgroup $H$ of $G$. Since any proper subgroup of $G_1$ is constant, we have that $g(a,0)$ generates a Kolchin-dense subgroup of $G_1$. Note that $$g(a^n,0)g(1,1)g(a^{-n},0) = g(1,a^n) \in G_2.$$ Since $a$ is transcendental over $\U^\partial$,  $1,a,a^2,\ldots,a^i$ are linearly independent over $\U^\partial$ for all $i$. Furthermore, since the $\U$-points of every proper differential algebraic subgroup of $\Ru(G) \cong \Ga$ form a finite-dimensional vector space over $\U^\partial$, $H$ contains a Kolchin-dense subset of $G_2$.  Therefore, the Kolchin-closure of $H$ is $G$.  Since the type of $G_2$ is $1$, $\tau(\Ru(G)) = 1$.
\end{example}

\section{Parameterized Picard--Vessiot Extensions}\label{sec:PPV}

\subsection{Definitions and Structure}  We briefly recall some of the definitions and results associated with the parameterized Picard--Vessiot theory.  Let $k$ be a $\Delta' = \{\partial, \partial_1, \ldots , \partial_m\}$-field and 
\begin{equation}\label{EQN:LDE}
\partial Y = AY, \quad A \in \Mn_n(k)
\end{equation}
a linear differential equation (with respect to $\partial$) over $k$.  A {\em parameterized Picard--Vessiot extension (PPV-extension)} $K$ of $k$ associated with \eqref{EQN:LDE} is a $\Delta'$-field $K \supset k$ such that there exists a $Z \in \GL_n(K)$ satisfying $\partial Z=AZ$, $K^\partial = k^\partial$, and   $K$ is generated over $k$ as a $\Delta'$-field by the entries of $Z$ (i.e., $K = k\langle Z\rangle$).  

The field $k^\partial$ is a $\Delta = \{\partial_1, \ldots, \partial_n\}$-field and, if it is differentially closed, a PPV-extension associated with \eqref{EQN:LDE} always exists and is unique up to $\Delta'$-$k$-isomorphism~\cite[Proposition~9.6]{PhyllisMichael}. Moreover, if $k^\partial$ is relatively differentially closed in $k$, then $K$ exists as well~\cite[Thm~2.5]{GGO} (although it may not be unique). Some other situations concerning  the existence of $K$ have been also treated in~\cite{Wibmer:existence}. 

If $K = k\langle Z\rangle$ is a PPV-extension of $k$, one defines the {\em parameterized Picard--Vessiot group (PPV-group)} of $K$ over $k$ to be
  $$
 G:= \{\sigma : K\to K\:|\: \sigma\text{ is a field automorphism, } \sigma\delta=\delta\sigma\ \forall \delta\in\Delta', \text{ and } \sigma(a)=a,\ \forall a\in k\}.
  $$
For any $\sigma \in G$, one can show that there exists a matrix $[\sigma] \in \GL_n\big(k^\partial\big)$ such that $\sigma(Z) = Z[\sigma]$ and the map $\sigma\mapsto [\sigma]$ is an isomorphism of $G$ onto a differential algebraic subgroup (with respect to $\Delta$) of $\GL_n\big(k^\partial\big)$. 

One can also develop the PPV-theory in the language of modules. 
A finite-dimensional vector space $M$ over the $\Delta'$-field $k$ together with a map $\partial : M \to M$ 
  is called a {\it parameterized differential module} if $$\partial(m_1+m_2) = \partial(m_1)+\partial(m_2)\ \ \text{and}\ \ \partial(am_1)=\partial(a)m_1+a\partial(m_1),\ \ m_1,m_2 \in M,\ a\in k.$$ 
 Let $\{e_1,\ldots,e_n\}$ be a $k$-basis  of $M$ and $a_{ij} \in L$ be such that $$\partial(e_i) = -\sum\nolimits_j a_{ji}e_j,\quad 1\Le i \Le n.$$ As in~\cite[\S1.2]{Michael}, for $v = v_1e_1+\ldots+v_ne_n$, 
 $$\partial(v) = 0\quad \Longleftrightarrow\quad \partial\begin{pmatrix}v_1\\
 \vdots\\
 v_n\end{pmatrix}= A\begin{pmatrix}v_1\\
 \vdots\\
 v_n\end{pmatrix},\quad  A := (a_{ij})_{i,j=1}^n.$$ Therefore, once we have selected a basis, we can associate a linear differential equation of the form \eqref{EQN:LDE} with $M$. Conversely, given such an equation, we define a map $$\partial:k^n\rightarrow k^n,\quad\partial(e_i) = -\sum\nolimits_j a_{ji}e_j,\quad A=(a_{ij})_{i,j=1}^n.$$  This  makes $k^n$ a parameterized differential module. The collection of parameterized differential modules over $k$ forms an abelian tensor category. In this category, one can define the notion of prolongation $M \mapsto P_i(M)$ similar to the notion of prolongation of a group action as in~\eqref{eq:prolongation}. 
 
 For example, if $\partial Y = AY$ is the differential equation associated with the module $M$, then (with respect to a suitable basis) the equation associated with $P_i(M)$ is
 $$
 \partial Y = \begin{pmatrix} A&\partial_i(A)\\ 0 & A \end{pmatrix}Y.
 $$
 Furthermore, if $Z$ is a solution matrix of $\partial Y = AZ$, then 
 $$
 \begin{pmatrix} Z & \partial_i(Z) \\ 0 & Z\end{pmatrix}
 $$
 satisfies this latter equation. Similar to the $s^{th}$ total prolongation of a group, we define the {\em $s^{th}$ total prolongation $P^s(M)$ of a module $M$}  as 
 $$
 P^s(M) = P_1^sP_2^s \cdots P_m^s(M).
 $$
 If $K$ is a PV-extension for \eqref{EQN:LDE}, one can define a $k^\partial$-vector space
 $$
 \omega(M) := \Ker(\partial: M\otimes_kK \to M\otimes_kK).
 $$
The correspondence $M \mapsto \omega(M)$ induces a functor $\omega$ (called a differential fiber functor) from the category of differential modules  to the category of finite-dimensional vector spaces over $k^\partial$ carrying the $P_i$ into the $P_i$ (see  \cite[Defs.~4.9,\,4.22]{GGO}, \cite[Definition~2]{OvchTannakian}, \cite[Definition~4.2.7]{Moshe}, \cite[Definition~4.12]{Moshe2012} for more formal definitions). Moreover, $$\big(\Rep_G,\mathrm{forget}\big)\cong\left(\big\langle P_1^{i_1}\cdot\ldots P_m^{i_m}(M)\:\big|\: i_1,\ldots,i_m\Ge 0\big\rangle_{\otimes},\omega\right)$$ as differential tensor categories \cite[Thms.~4.27,\,5.1]{GGO}. This equivalence will be further used in \S\ref{sec:algorithm} to help explain the algorithm that calculates the PPV-group of a PPV-extension of type $0$. Before we describe this algorithm, we will describe the relation between the PPV-theory and the usual PV-theory \cite{Kol,Michael} and give a more detailed description of the PPV-group.

Let $K = k\langle Z \rangle$ be a PPV-extension of $k$ with $\partial Z = AZ$  and assume that $k^\partial$ is a $\Delta$-differentially closed field.  We begin by recalling the structure of $K$.  Recall that $$\Delta' = \{\partial\} \cup \Delta,$$ where $\Delta = \{\partial_1, \ldots , \partial_m\}$. Let $Y$ be an $n\times n$ matrix of $\Delta$-differential indeterminates and  $k\big\{Y, \frac{1}{Y}\big\}_\Delta$ be the ring of $\Delta$-differential polynomials. We can extend (i.e., prolong) $\partial$ to this ring by setting $$\partial Y = AY\quad\text{and}\quad \partial(\theta Y) = \theta(AY)\ \text{for any}\  \theta \in \Theta.$$ The $\Delta'$-ring $k\left\{Z,\frac{1}{\det Z}\right\}$ is called a PPV-ring for $\partial Y = AY$, and we may write this as
$$ 
k\{Z,1/\det Z\} = k\{Y, 1/\det Y\}_\Delta/I,
$$
where  $I$ is a maximal (and, therefore, prime) $\Delta'$-ideal. As mentioned above, $I$ is  the radical of a finitely generated differential ideal, that is, $$I = \sqrt{[p_1, \ldots, p_\ell]},\quad p_i \in  k\{Y,1/\det Y\}_\Delta.$$ Given $g \in \GL_n(k)$, the map $Y \mapsto Yg$ induces a differential isomorphism $$\phi_g:k\{Y, 1/\det Y\} \rightarrow k\{Y, 1/\det Y\}.$$ It is clear \cite[p.~146]{PhyllisMichael} that the PPV-group $G$ can be described as
$$
G = \left\{ g \in \GL_n\big(k^\partial\big) \ \big| \ \phi_g(I) \subset I\right\}.
$$
For any nonegative integer $s$, let $$\Psi(s) = \left\{\theta = \partial_1^{i_1}\partial_2^{i_2}\cdot\ldots\cdot \partial_m^{i_m} \in \Theta, \  i_j \Le s, \ j = 1, \ldots , m \right\}$$
and $$\quad R_s = k[\theta(Y), 1/\det Y : \theta \in \Psi(s)].$$
 For $g \in \GL_n\big(k^\partial\big)$, we may restrict $\phi_g$ to $R_s$, and this is the automorphism induced by  the $s^{th}$ total prolongation of $\GL_n$ (see \S\ref{sec:rep}).
Let 
$$
I_s = I\cap R_s\quad \text{and}\quad 
 K_s = k\big(\theta(Z) : \theta \in \Psi(s)\big) = \Quot(R_s/I_s).
 $$  
One can show, as in \cite[Proposition~6.21]{CharlotteMichael}, or, with a weaker assumption on $k^\partial$ but restricted to $|\Delta| =1$ in \cite{Wibmer:existence},  that $K_s$ is a Picard--Vessiot extension of $k$ corresponding to the $s^{th}$ total prolongation of the equation $\partial Y = AY$. Furthermore, the restriction of elements of the PPV-group $G$ to $K_s$ yields a group that is Zariski-dense in the PV-group $H$ of $K_s$ over $k$. The PV-group $H$ of $K_s$ over $k$ can be identified with a Zariski-closed subgroup of $\GL_{N_s}\big(k^\partial\big)$. For $$h \in \GL_{N_s}\big(k^\partial\big)\quad \text{and}\quad q \in k[\theta(Y), 1/\det Y : \theta \in \Psi(s)],$$ let $q^h$ denote the polynomial resulting from the change of variables induced by $h$. We then can identify $H$ as 
$$
H= \left\{ h \in \GL_{N_s}\big(k^\partial\big) \ \big| \ I_s^h \subset I_s\right\}.
$$
We now will show the following result.  
\begin{proposition}\label{prop:defeq} Let $k$, $K$, and $\{p_i\}$ be as above.  If $s \Ge \max\{\ord p_i\}$, then  the PPV-group of $K$ over $k$ is equal to 
$$
G = \left\{ g \in \GL_n\big(k^\partial\big) \ \big| \ P^s(g) \in H \right\},
$$
where $H \subset \GL_{N_s}\big(k^\partial\big)$ is the PV-group of $K_s$ over $k$ and $P^s(g)$ denotes the $s^{th}$ total prolongation of $g$.
\end{proposition}
\begin{proof}
The discussion preceding the statement of the proposition shows that, if $g$ is an element of the PPV-group, then $g \in G$.  Conversely, if $P^s(g) \in H$, then $P^s(g)$ leaves $I_s$ invariant. Since $I_s$ contains the (radical) generators of $I$, $\phi_g$ leaves $I$ invariant.
Therefore, $g$ is in the PPV-group of $K$ over $k$.
\end{proof} 

This proposition implies that, for $s$ as in the proposition,  once we know the defining equations of the PV-group of $K_s$ over $k$, we can find defining equations for the PPV-group of $K$ over $k$.  Note that, if we can find a bound $s$ on the orders of the $p_i$ generating $I$ as above, then we have reduced the problem of finding the PPV-group of $K$ over $k$ to the problem of finding the PV-group of $K_s$ over $k$. 

 For certain fields $k$, Hrushovski~\cite{Hrushovski} has solved this latter problem.  In general, we do not know how to find such bound $s$ on the orders of the $p_i$ but we will show in \S\ref{sec:algorithm} that we can find such an $s$ when $K$ has differential type $0$ over $k$. This will depend on the following result.

\begin{proposition}\label{prop:zeroppv} Let $k$, $K$, $s$, $K_s$, $I$, and $I_s$ be as above. If $s$ is an integer such that $K_s = K_{s+1}$, then $$\big[I_{s+1}\big]_\Delta = I,$$ where $[I_{s+1}]_\Delta$ is the $\Delta$-differential ideal generated by $I_{s+1}$. In particular $I_s$ contains the (differential and, therefore, radical differential) generators of $I$. Furthermore, the differential type of of $K$ over $k$  is $0$ and the differential type of the PPV-group $G$ is also $0$. 

Conversely, if the differential type of $G$ is $0$, then the differential type of  $K$ over $k$ is $0$ and there exists an $s$ such that $K_s = K_{s+1}$. \end{proposition} 

\begin{proof} To simplify the notation, we shall assume that $m=1$, that is,  $\Delta = \{\partial_1\}$ and $\Delta' = \{\partial, \partial_1\}$. Let $K = k\langle Z \rangle$, and assume that $K_s= K_{s+1}$.  This implies that $$\partial_1^{s+1}(Z) \in K_{s}.$$ The entries of  $\partial_1^{s+1}(Z)$ satisfy scalar linear differential equations over $k$, and so each entry of $\partial_1^{s+1}(Z)$ lies in the PV-ring in $K_s$ \cite[Corollary~1.38]{Michael}. This latter ring is $$k\big[Z, \partial_1Z, \ldots ,\partial_1^sZ, 1/\det Z\big].$$ In particular, for each entry $\partial_1^{s+1}z$ of $\partial_1^{s+1}(Z)$, there exists a polynomial $$p_z\big(Y, \partial_1Y, \ldots , \partial_1^sY, 1/\det Y\big)$$ such that 
\begin{equation}\label{eqn:lin}\partial_1^{s+1}z = p_z\big(Z, \partial_1Z, \ldots , \partial_1^sZ, 1/\det Z\big).
\end{equation}

We now claim that the differential ideal $J = [I_{s+1}]_\Delta$, the $\Delta$-ideal generated by $I_{s+1}$,  is a maximal $\Delta'$-ideal in $k\big\{Y, \frac{1}{Y}\big\}_\Delta$. Note that $Z$ is a zero of $J$, so $J$ is a proper ideal.  Since $I_{s+1}$ is a $\{\partial\}$-ideal, one can show that $J$ is also a $\{\partial\}$-ideal. Differentiating \eqref{eqn:lin} sufficiently many times and eliminating higher derivatives, one sees that, for each $t > s$ and each entry $z$ of $Z$, $J$ contains a polynomial of the form 
$$\partial_1^t y - p_{z,i}\big(Y, \partial_1Y, \ldots , \partial_1^sY, 1/\det Y\big),$$
where $$p_{z,i}\big(Y, \partial_1Y, \ldots , \partial_1^sY, 1/\det Y\big) \in k\big[Y, \partial_1Y, \ldots ,\partial_1^sY, 1/\det Y\big].$$ Therefore, for any differential polynomial $f$, there exists a polynomial $$q_f\big(Y, \partial_1Y, \ldots , \partial_1^sY, 1/\det Y\big)$$ such that $f-q_f \in J$. Let $J'$ be a differential ideal containing $J$ and let $f \in J'$.  Since $I_{s+1}$ is a maximal $\{\partial\}$-ideal, we must have $$q_f \in J'\cap k\big[Y, \partial_1Y, \ldots ,\partial_1^sY, 1/\det Y\big] = I_{s+1}.$$  Therefore, $f \in J$, and we have shown that $J$ is maximal.   Since $J \subset I$, we must have $J = I$.

One sees from the above proof that $K_s = K$, and so $\trdeg_kK$ is finite. Therefore, the type of $K$ over $k$ is $0$. We know that the differential transcendence degree of $K$ over $k$ is the same as the differential transcendence degree of $G$ over $k^\partial$ \cite[Proposition~6.26]{CharlotteMichael}.  Therefore, $\tau(G)=0$.

If  $\tau(G)=0$, then \cite[Proposition~6.26]{CharlotteMichael} also implies that the differential transcendence degree of $K$ over $k$ is $0$, and so $K$ has differential type $0$ as well.  Hence, since $K$ is finitely generated over $k$ as a differential field, it must be finitely generated over $k$ in the usual algebraic sense. Therefore, for some $s$, we must have $$K_s = K_{s+1} = \ldots  = K.$$
The proof of the proposition when $m>1$ follows in a similar fashion noting that, for any $\theta \in \Psi(s)$, $\partial_i \in \Delta$ and any entry $z \in Z$, there exists a polynomial $$p_{z,\theta,i} \in k\big[\big\{\theta' Y: \theta' \in \Psi(s)\big\}, 1/\det Y\big]$$ such that $$\partial_i(\theta(z))  = p_{z,\theta,i}\big(\big\{\theta' Z: \theta' \in \Psi(s)\big\}, 1/\det Z\big).\qedhere$$ \end{proof}

\subsection{Algorithmic Considerations}\label{sec:algorithm}
Let $\U$ be a $\Delta = \{\partial_1, \ldots, \partial_m\}$-differentially closed field and let $k = \U(x)$ be a $\Delta'$-field where $$\Delta' = \{\partial\}\cup \Delta,\quad \partial(x) = 1, \ \partial(u) = 0 \ \ \forall u \in \U,\quad \text{and}\quad\partial_i(x) = 0,\ i= 1, \ldots, m.$$ We shall further assume that $\U$ is a computable differential field in the sense that one can effectively perform the arithmetic operations and apply the $\partial_i$. We shall show that there are algorithms to solve the following problems:
\begin{itemize} 
\item Given a linear differential equation $\partial Y = AY$ with $A \in \Mn_n(k)$ and whose PPV-group $G \subset\GL_n(\U)$ has differential type $0$, find defining equations for $G$.
\item Given a linear differential equation $\partial Y = AY$ with $A \in \Mn_n(k)$, decide if its PPV-group $G$ has the property that $G/\Ru(G)$ is constant.
\end{itemize}
Of course, if the answer to the second problem is positive, Theorem~\ref{thm:Main} implies that $\tau(G)=0$ and so one could use the first algorithm to calculate defining equations of $G$.
\subsubsection{Algorithm 1.}\label{alg1}Let $k$ and $A$ be as above and assume that the PPV-group $G$ of this equation has type $0$.  We will present an algorithm that computes the defining equations of $G$. Two of the main tools will be Propositions~\ref{prop:defeq} and \ref{prop:zeroppv}, and we will use the notation used in these propositions. Let $M$ be the parameterized differential module associated with  $\partial Y = AY$ and $K$ be the associated PPV-extension. 

The usual PV-extension of $k$ associated with the $s^{th}$ total prolongation $P^s(M)$  (now considered just as a differential module with respect to $\partial$) is $K_s$, and we denote the PV-group of $K_s$ over $k$ by $G_s$. Since $K_s \subset K_{s+1}$ and both are PV-extensions of $k$, the usual PV-theory \cite[Proposition~1.34]{Michael} implies that $G_s$ is a homomorphic image of $G_{s+1}$. The homomorphism can be made quite explicit.

 If $\partial Y = A_s Y$ is the equation associated with $P^s(M)$, one can select a basis of $P^{s+1}(M)$ so that  the associated equation $\partial Y = A_{s+1}Y$ has $A_{s+1}$ as a block triangular matrix with blocks of $A_s$ on the diagonal. The elements of $G_{s+1}$ will have a similar form to $A$ and the homomorphism $$\pi_{s+1}:G_{s+1} \rightarrow G_s$$ is given by mapping an element $g$ of $G_{s+1}$ onto the block appearing in the upper left-hand corner of $g$. Again, from the PV-theory \cite[Proposition~1.34]{Michael}, we know that $K_s = K_{s+1}$ if and only if the surjective projection $\pi_{s+1}$ is injective. 

The algorithm proceeds as follows. For each $s = 0, 1, \dots$, we  use the algorithm of \cite{Hrushovski} to successively calculate the defining equations of $G_s$ and $G_{s+1}$. We then use elimination theory to decide if the projection $\pi_{s+1}$ is injective. Since we are assuming $\tau(G) = 0$, Proposition~\ref{prop:zeroppv} implies that, for some $s$, $K_s = K_{s+1}$, and so, for this $s$, the homomorphism $\pi_{s+1}$ will be injective. Propositions~\ref{prop:defeq} and \ref{prop:zeroppv} tell us that generators of the  defining ideal of $G_{s+1}$ give us the defining equations for $G$.  

We note that one can apply  the above method without knowing in advance if the PPV-group has differential type $0$.  If the method does terminate, then we will know that the PPV-group has differential type $0$, and we will have the PPV-group. Examples are given below.

This algorithm may also be approached  via  the Tannakian theory. Let $M$ be a parameterized differential module over $k$ and let $\langle M\rangle_{\otimes}$ denote the smallest rigid abelian tensor category containing $M$. That is,  $\langle M\rangle_{\otimes}$ is obtained from $M$ by successively applying the operations of linear algebra: $\otimes$, $\oplus$, duals, and subquotients. Let $G$ be the PPV-group of $M$ and assume that $\tau(G) = 0$. This implies that
the coordinate ring of $G$ is a finitely generated $\U$-algebra \cite[Theorem~II.13.7]{Kol}. By \cite[Proposition~2.20]{Deligne}, the tensor category $\Rep_G$ admits one generator $N$. Since $$\Rep_G = \big\langle P_1^{i_1}\cdot\ldots\cdot P_m^{i_m}(M)\:|\: i_1,\ldots,i_m\Ge 0\big\rangle_{\otimes}=\bigcup_{(i_1,\ldots,i_m)} \big\langle P_1^{i_1}\cdot\ldots\cdot P_m^{i_m}(M)\big\rangle_{\otimes},$$
there exist  $p_1, \ldots, p_m \Ge 0$ such that
$$N\in \left\langle P_1^{p_1}\cdot\ldots\cdot P_m^{p_m}(M)\right\rangle_{\otimes}.$$
Therefore, for all $(i_1,\ldots,i_m)$ with $i_j > p_j$, $1\Le j\Le m$,  $$P_1^{i_1}\cdot\ldots\cdot P_m^{i_m}(M)\in \left\langle P_1^{p_1}\cdot\ldots\cdot P_m^{p_m}(M)\right\rangle_{\otimes}.$$
To achieve this containment, 
order  $m$-tuples of the form $(i_1,\ldots,i_m)$ degree-lexicographically. We then, following this enumeration of the $m$-tuples by natural numbers, let
 $$M_i := P_1^{i_1+1}\cdot\ldots\cdot P_m^{i_m+1}(M)\quad\text{and}\quad N_i := P_1^{i_1}\cdot\ldots\cdot P_m^{i_m}(M).$$ 
 At each iteration, we verify whether
 \begin{equation}\label{eq:algstep}
M_i \in {\langle N_i\rangle}_\otimes
\end{equation}
by checking whether the PV Galois group $H_i$ of $N_i$ coincides with the projection of the PV Galois group $H_i'$ of $M_i$ ($N_i$ is a submodule of $M_i$), which can be done using~\cite{Hrushovski}. Once we have found $i$ such that~\eqref{eq:algstep} holds, we output $H_i$ (keeping in mind what is the derivative of what
among the indeterminates in the coordinate ring of $H_i'$, out of which one can extract $G$).

\subsubsection{Algorithm 2.}\label{alg2}  Once again, let $k$ and $A$ be as above. We present an algorithm to determine if the PPV-group $G$ of $\partial Y = AY$ has the property that $G/\Ru(G)$ is constant.  The first step is to factor $\partial Y = AY$, that is, to find a $W \in \GL_n(k)$ such that the matrix $$B = \partial(W) W^{-1} + WAW^{-1}$$ is in block-upper triangular form 
$$ B = \begin{pmatrix} B_1 & * & * & *& *\\0 & B_2 & * & * & *\\ \vdots& \vdots& \vdots& \vdots & \vdots\\ 0 & 0& \ldots &0 &B_t\end{pmatrix}$$
 and, if $B_i$ is one of the blocks on the diagonal, then the equation $\partial Y = B_i Y$ is irreducible (cannot be factored further). Algorithms to perform such a factorization can be found in \cite{grigoriev90,grigoriev90b}; see \cite[Ch.~4.2]{Michael} for other references. The new equation $\partial Y = BY$ has the same PPV-group $G$ as $\partial Y = AY$. Furthermore, the PPV-group of $\partial Y = B_{\rm diag} Y$, where 
 $$ B_{\rm diag} = \begin{pmatrix} B_1 & 0 & \ldots  & \ldots & 0\\0 & B_2 & 0& \ldots & 0\\ \vdots& \vdots& \vdots& \vdots & \vdots\\ 0 & 0& \ldots &0 &B_t\end{pmatrix},$$
 is precisely $G/\Ru(G)$ (see \cite[Proposition~4.2]{CompMichael}). To verify that $G/\Ru(G)$ is constant,
it is enough to verify that $B_{\diag}$ is completely integrable \cite[Proposition~3.9]{PhyllisMichael}. It follows from \cite[Theorem~6.3, Rmk.~6.4]{GO} that, to do this, it is enough to show that, for all $i$, $1\Le i\Le m$,
there exists $C_i\in \Mn_n(\U(x))$ such that
\begin{equation}\label{eq:intcond}
\partial_i(B_{\diag}) - \partial(C_i) = \big[C_i,B_{\diag}\big].
\end{equation}  This is a problem of finding rational solutions of linear differential equations, and there are algorithms to do this (see \cite[Ch. 4.1]{Michael}; there are also implementations in {\sc Maple}).  Such matrices $C_i$ exist if and only if   $G/\Ru(G)$ is constant. For examples illustrating the process of deciding if the $C_i$ exist, see \cite{GO}. 

We note that, combining Algorithms~1 and~2, we see that there is an algorithm to decide, for any given $$\partial Y = AY,\ A \in \Mn_n(\U(x)),$$ with PPV-group $G$, if $G/\Ru(G)$ is constant and, if so, calculate $G$ (since, in this case, $\tau(G) = 0$).

\subsection{Examples}
We will now illustrate how Algorithm 1 works on concrete examples. In all of these examples, we will take $\U$ to be the differential closure of the field $\Q(t)$ with $\Delta = \{\partial_t\}$ and will use the notation $\partial_x$ for $\partial$ and $\partial_t$ for $\partial_1$ when discussing the field $\U(x)$. Although Algorithm 1, as stated, relies on Hrushovski's algorithm, we will use {\it ad hoc} methods  to calculate PV-groups.
\begin{example}Consider the equation (see also \cite[Example~3.1]{PhyllisMichael})
\begin{equation}\label{eq:txy}
\partial_x(y)=\frac{t}{x}y,
\end{equation}
over the field $\U(x)$, and let $M$ be the corresponding differential module. We start by applying Algorithm 2, that is, we wish to determine if there exists an element $c \in \U(x)$ such that $$\partial_xc = \partial_t (t/x) =1/x.$$ Since the residues at any pole of $\partial_xc$ would be zero, such a $c$ cannot exist. Therefore, if $G$ is the PPV-group of this equation, we have that $G/\Ru(G)$ is not constant.  Nonetheless, we will apply the method of Algorithm 1 and see that it halts.

 The PV-group of this equation is a subgroup of $\Gm(\U)$.  Since, for all $n \ne 0$, $$\partial_x(y)=\frac{nt}{x}y$$ has no solutions in $\U(x)$, the PV-group of $M$ is $G_0 = \Gm$ \cite[Example~1.19]{Michael}. The first prolongation $P^1(M)$ is given by the matrix
$$
P^1(A):=\begin{pmatrix}
t/x&1/x\\
0&t/x
\end{pmatrix}.
$$
Let $G_1$ be the PV-group of this equation. One sees that $G_1$ is a subgroup of 
$$\ \left\{ \begin{pmatrix}a & b\\
0& a
\end{pmatrix}\:\Big|\: a,\,b \in \U, a\neq 0\right\} \simeq  \Gm \times \Ga$$
that  projects onto $\Ga$.  Therefore, either $G_1 = \Gm$ or $G_1 = \Gm \times \Ga$.  If  $G_1= \Gm$, then, by \cite[Corollary~1.32]{Michael}, $$\partial_xY = P^1(A)Y$$ would be equivalent to an equation $\partial_xY = BY$, where $B$ is in the Lie algebra of diagonal matrices (that is, there would exist $W \in \GL_2(\U(x))$ such that $$B = \partial_x(W) W^{-1} + W P^1(A)W^{-1}$$ is diagonal).  A calculation in {\sc Maple} shows that this is not possible. Therefore, $$G_1=\Gm \times \Ga.$$
The second prolongation $P^2(M)$ is associates with the equation $\partial_xY = P^2(A)$, where 
$$P^2(A) = \begin{pmatrix}
t/x&1/x& 1/x &0\\
0&t/x&0&1/x\\
0&0&t/x&1/x\\
0&0&0&t/x
\end{pmatrix}.$$
Observe that $P^2(A)$ is contained in the commutative Lie algebra
$$
\left\{\left.\begin{pmatrix}u & v&v&0\\
0& u&0&v\\
0&0&u&v\\
0&0&0&u
\end{pmatrix}\:\right|\: u,\,v \in \U\right\}
$$
and that this is the Lie algebra of the algebraic group 
$$
H = \left\{\left.\begin{pmatrix}a & b&b&c\\
0& a&0&b\\
0&0&a&b\\
0&0&0&a
\end{pmatrix}\:\right|\: a,\,b, c \in \U, a\neq 0, ac-b^2=0\right\} \simeq \Gm \times \Ga.
$$
Therefore, \cite[Proposition~1.31]{Michael} implies that the PV-group of $\partial_xY = P^2(A)$ is a subgroup of $H$ that projects onto $G_1$. One sees that this implies that $G_2 = H$ and the projection $\pi_2:G_2\rightarrow G_1$ is injective. Propositions~\ref{prop:defeq} and~\ref{prop:zeroppv} imply that the PPV-group $G$ of \eqref{eq:txy} is
$$ G = \left\{g \in \Gm(\U) \ \left| \ \begin{pmatrix}g & \partial_tg&\partial_tg& \partial_t^2g\\\
0&g&0&\partial_tg\\
0&0&g&\partial_tg\\
0&0&0&g
\end{pmatrix}\right. \in H \right\}.
$$
Examining the defining equations of $H$ (substituting $g$ for $a$, $\partial_tg$ for $b$, and $\partial_t^2g$ for $c$), we see that 
$$
G = \left\{g\in \Gm \ \big| \ g\big(\partial_t^2g\big) - (\partial_t g)^2 = 0\right\}.
$$
Note that one can also obtain the above equation by calculating the matrix of $$P^1(M)\otimes_k P^1(M)\otimes_k M^\vee.$$ 
\end{example}

\begin{example}
Consider the equation $\partial_xY = AY$ where
$$A = \begin{pmatrix}
1& \frac{t}{x}+\frac{1}{x+1}\\
0&1
\end{pmatrix}.
$$
Let $M$ be the corresponding differential module. We again begin by applying Algorithm~2. The equation is already factored, and its diagonal is constant so can be easily seen to satisfy the conditions of this algorithm.  Therefore, if $G$ is the PPV-group, we have that $G/\Ru(G)$ is constant and we have determined that the PPV-group has differential type $0$. We now know that Algorithm 2 will halt with the correct answer. 

The PV-extension of $k=\U(x)$ corresponding to this equation is $$K_0 = k\left(e^x, t\log(x)+\log(x+1)\right).$$  From \cite[Example~1.19]{Michael} and \cite[Example~1.18]{Michael}, one sees that the PV-groups of $ k(e^x)$ and $k( t\log(x)+\log(x+1))$ are $\Gm(\U)$ and $\Ga(\U)$, respectively. Since these are quotients of the PV-group of  $\partial_xY = AY$, one sees that the PV-group of this equation is 
$$
G_0 = 
\ \left\{\begin{pmatrix}a & b\\
0& a
\end{pmatrix}\:\Big|\: a,\,b \in \U, a\neq 0\right\} \simeq  \Gm \times \Ga.
$$
The first prolongation $P^1(M) $ of $M$ is associated with the equation $\partial_xY = P^1(A)$, where 
$$P^1(A):= \begin{pmatrix}
1& \frac{t}{x}+\frac{1}{x+1}&0&\frac{1}{x}\\
0&1&0&0\\
0&0&1& \frac{t}{x}+\frac{1}{x+1}\\
0&0&0&1
\end{pmatrix}.
$$
This matrix lies in the Lie algebra of the algebraic group
$$
H_1 = \left\{\left.\begin{pmatrix}a & b&0&c\\
0& a&0&0\\
0&0&a&b\\
0&0&0&a
\end{pmatrix}\:\right|\: a,\,b,\,c \in \U,\, a\ne 0\right\} \simeq \Gm\times\Ga\times\Ga.
$$
This again implies that the associated PV-group is a subgroup of this group. The associate PV-extension is $$K_1 = k(e^x, t\log(x)+\log(x+1), \log(x)).$$ We claim that $\trdeg(K_1/k) = 3$. If not,  the Kolchin-Ostrowski theorem \cite[\S2]{Kolchin1968} implies that: 
 \begin{enumerate}
 \item Either there exists $0\ne n \in \ZZ$ such that ${(e^x)}^n \in \U(x)$. This would imply that the differential equation $\partial_x(y)=ny$ has a solution in $\U(x)$, which is impossible;
 \item Or there exist $c_1,\,c_2 \in \U$ such that $$c_1 (t\log x + \log(x+1)) + c_2\log x \in \U(x).$$ This would imply that there exists $f \in \U(x)$ such that $$\frac{c_1t+c_2}{x}+\frac{c_1}{(x+1)} = \partial_xf,$$
 which is impossible as well. 
 \end{enumerate}
 Since $H_1$ is connected, the dimension of $H_1$ is $3$, and the dimension of $G_1$ is equal to $\trdeg(K_1/k) = 3$, we must have $G_1 = H_1$.  In particular, the projection of $G_1$ onto $G_0$ cannot be injective.
 
 The second prolongation $P^2(M) $  of $M$ is associated with the matrix 
 $$
 P^2(A) =
\begin{pmatrix}
1& \frac{t}{x}+\frac{1}{x+1}&0&\frac{1}{x}&0&\frac{1}{x}&0&0\\
0&1&0&0&0&0&0&0\\
0&0&1& \frac{t}{x}+\frac{1}{x+1}&0&0&0&\frac{1}{x}\\
0&0&0&1&0&0&0&0\\
0&0&0&0&1& \frac{t}{x}+\frac{1}{x+1}&0&\frac{1}{x}\\
0&0&0&0&0&1&0&0\\
0&0&0&0&0&0&1& \frac{t}{x}+\frac{1}{x+1}\\
0&0&0&0&0&0&0&1
\end{pmatrix},
$$
Note that $P^2(A)$ belongs to the Lie algebra of the algebraic group
$$
H_2 = \left\{\left.\begin{pmatrix}
a& b&0&c&0&c&0&0\\
0&a&0&0&0&0&0&0\\
0&0&a& b&0&0&0&c\\
0&0&0&a&0&0&0&0\\
0&0&0&0&a& b&0&c\\
0&0&0&0&0&a&0&0\\
0&0&0&0&0&0&a& b\\
0&0&0&0&0&0&0&a
\end{pmatrix}\ \right| \ a,\,b,\,c \in \U,\, a\ne 0\right\} \simeq \Gm\times\Ga\times\Ga.
$$
Therefore, the PV-group $G_2$ associated with $P^2(M)$ is a subgroup of $H_2$ that projects surjectively onto $G_1$. The only possibility is that $G_2 = H_2$. Note that this projection is injective so $K_1 = K_2$. Propositions~\ref{prop:defeq} and \ref{prop:zeroppv} imply that the PPV-group $G$ of \eqref{eq:txy} is
$$
G= \left\{\begin{pmatrix} e & f\\0&e \end{pmatrix} \in \Gm(\U)\times \Ga(\U)\ \left| \ \begin{pmatrix}
e& f&\partial_t e&\partial_t f&\partial_t e&\partial_t f&d_1^2 e&\partial_t^2 f\\
0&e&0&\partial_te&0&\partial_te&0&\partial_t^2e\\
0&0&e& f&0&0&\partial_t e&\partial_t f\\
0&0&0&e&0&0&0&\partial_te\\
0&0&0&0&e& f&\partial_t e&\partial_t f\\
0&0&0&0&0&e&0&\partial_te\\
0&0&0&0&0&0&e& f\\
0&0&0&0&0&0&0&e
\end{pmatrix}\right.\in H_2\right\}.
$$
Examining the defining equations of $G_2$, we see that 
$$G= \left\{\left.\begin{pmatrix} e & f\\0&e \end{pmatrix} \in \Gm(\U)\times \Ga(\U)\ \right| \ \partial_te= 0, \  \partial_t^2f = 0 \right\}.
$$
 \end{example}
 
 \begin{example}[Picard--Fuchs equation] As in \cite[Example~6.9]{GO}, consider $K=\U(x,z)$, where $z^2=x(x-1)(x-t)$. Following \cite[Example~2]{Carlos}, let 
$$p = -\frac{1}{2}\left(\frac{1}{x}+\frac{1}{x-1}+\frac{1}{x-t}\right).$$  We will consider the parameterized linear differential equation
\begin{equation}\label{eq:PF}
\partial^2(y)-p\partial(y)=0
\end{equation}
over $k= \U(x)$ and  outline how Algorithm 1 can be used to calculate the PPV-group of this equation. A calculation in {\sc Maple} using
the procedure {\tt kovacicsols} of the {\tt DEtools} package shows that the PV-extension corresponding to this equation is $$K_0 = k\left(z, \int\frac{1}{z} dx\right).$$ Since $\partial_xy = \frac{1}{z}$ has no solution in $k(z)$, the Kolchin-Ostrowski Theorem implies that the PV-group of $K_0$ over $k(z)$ is $\Ga(\U)$.  A simple calculation shows that the PV-group of $K_0$ over $k$ is $$\left\{\left.\begin{pmatrix}1&b\\0&a\end{pmatrix}\ \right|\ a^2=1,\ b \in \U\right\} \cong \ZZ/2\ZZ\ltimes\Ga$$
for the fundamental solution matrix of~\eqref{eq:PF} chosen as
$$
\begin{pmatrix}
1&\int\frac{1}{z} dx\\
0&\frac{1}{z}
\end{pmatrix}.
$$
The PV-extension of the first prolongation is 
$$K_1 = k\left(z, \int\frac{1}{z} dx, \int \partial_t\left(\frac{1}{z}\right) dx\right).$$
A calculation shows that, if $c_1,c_2 \in \U$ and $$c_1 \frac{1}{z} + c_2 \partial_t\left(\frac{1}{z}\right) = \partial_x f$$ for some $f \in \U(x,z)$, then $c_1 = c_2 = 0$. Therefore, $\int\frac{1}{z} dx$ and $\int \partial_t(\frac{1}{z}) dx$ are algebraically independent over $\U(x,z)$.  This implies that the PV-group of $K_1$ over $k$ is $$ \ZZ/2\ZZ\ltimes\Ga \times \Ga.$$ The PV-extension of the second prolongation is
$$K_2 = k\left(z, \int\frac{1}{z} dx, \int \partial_t\left(\frac{1}{z}\right) dx, \int \partial_t^2\left(\frac{1}{z}\right) dx\right).$$
This implies that the PV-group of $K_2$ over $k$ is a subgroup of $$ \ZZ/2\ZZ\ltimes\Ga \times \Ga\times\Ga.$$
As in \cite[\S12]{manin58}, one can show that 
$$2t(t-1)\partial_t^2\left(\frac{1}{z}\right)+(4t-2)\partial_t\left(\frac{1}{z}\right)+\frac{1}{2}\frac{1}{z} = \partial_x(f)$$
for some $f \in \U(x,z)$. In particular, $K_2 = K_1$, and the PV-group of $K_2$ over $k$ can be identified with 
$$ \left\{(a,b,c,d) \in \ZZ/2\ZZ\ltimes\Ga \times \Ga\times\Ga \ \Big| \ a^2 = 1 \text{ and } \ 2t(t-1)d+(4t-2)c+\frac{1}{2}b = 0\right\}.$$
The projection of this group onto the PV-group of $K_1$ is injective. So, the PPV-group of \eqref{eq:PF} is 
$$\left\{\left.\begin{pmatrix}1&b\\0&a\end{pmatrix}\ \right|\ a^2=1\ \text{ and }\ 2t(t-1)d_t^2b+(4t-2)\partial_tb+\frac{1}{2}b = 0\right\} \subset \ZZ/2\ZZ\ltimes\Ga.$$
\end{example}

\subsection*{Acknowledgements} We are grateful to the referee for the very helpful suggestions.

\bibliographystyle{spmpsci}
\small
\bibliography{unipotent}

\end{document}